\documentclass[10pt,leqno]{amsart}

\usepackage{amsmath}
\usepackage{amsfonts}
\usepackage{amssymb}
\usepackage{mathtools}
\usepackage{graphicx}
\usepackage{color}
\usepackage{hyperref}
\usepackage{xcolor}
\newtheorem{theorem}{Theorem}[section]
\newtheorem*{theorem*}{Theorem}
\newtheorem{lemma}[theorem]{Lemma}
\newtheorem{corollary}[theorem]{Corollary}
\newtheorem{proposition}[theorem]{Proposition}
\newtheorem{conjecture}[theorem]{Conjecture}
\newtheorem{maintheorem}{Theorem}[]

\theoremstyle{definition}
\newtheorem{definition}[theorem]{Definition}
\newtheorem{example}[theorem]{Example}

\newtheorem{problem}{Problem}

\newtheorem{remark}[theorem]{Remark}

\newtheorem{question}[problem]{Question}

\newcommand{\od}{\odot}

\newcommand{\R}{\mathbb{R}}

\def\Ric{\text{Ric}}

\def\C{\mathbb{C}}
\def\R{\mathbb{R}}

\def\CP{{\mathbb{CP}}}

\def\id{\operatorname{id}}
\def\Ric{\operatorname{Ric}}

\def\tr{\operatorname{tr}}

\numberwithin{equation}{section}
\makeatletter
\newcommand*\owedge{\mathpalette\@owedge\relax}
\newcommand*\@owedge[1]{%
  \mathbin{%
    \ooalign{%
      $#1\m@th\bigcirc$\cr
      \hidewidth$#1\m@th\wedge$\hidewidth\cr
    }%
  }%
}
\makeatother

\begin{document}

\title[Characterizing $\mathbb{CP}^n$ via the Calabi Curvature Operator]{A Characterization of Complex Projective Space via the Calabi Curvature Operator}

\author[El-Hasan]{Hasan M. El-Hasan}
\address{University of California, Santa Barbara, South Hall, CA, 93106}
\email{elhasan@ucsb.edu}
\author[Li]{Xiaolong Li}
\address{Department of Mathematics and Statistics, Auburn University, Auburn, AL, 36849, USA}
\email{xil0005@auburn.edu}
\author[Nienhaus]{Jan Nienhaus}
\address{Department of Mathematics, KU Leuven, Celestijnenlaan 200A, Leuven B-3000, Belgium}
\email{jan.nienhaus@kuleuven.be}
\author[Petersen]{Peter Petersen}
\address{University of California, Los Angeles, 520 Portola Plaza, CA, 90095}
\email{petersen@math.ucla.edu}
\author[Stanfield]{James Stanfield}
 \address{University of Wollongong, Northfields Avenue, Wollongong, NSW 2500, Australia}
 \email{jstanfield@uow.edu.au} 
\author[Wink]{Matthias Wink}
\address{University of California, Santa Barbara, South Hall, CA, 93106}
\email{wink@math.ucsb.edu}




\begin{abstract}
    We prove a Tachibana-type result for K\"ahler-Einstein manifolds isolating complex projective space based on its Calabi curvature operator. The proof uses a new Bochner formula expressing the Lichnerowicz curvature term via the Calabi curvature operator. This resolves one of the problems \cite{Bochner} posed at the AIM workshop ``The Bochner Technique'' (May 2026). 
\end{abstract}

\maketitle

\section{Introduction}

The main objective of this paper is to prove a rigidity theorem characterizing complex projective space among closed K\"ahler-Einstein manifolds using the Calabi curvature operator. To define the Calabi curvature operator of a K\"ahler manifold, consider the bundle $S^{2,0}\subset T^{1,0} \otimes T^{1,0}$ of symmetric tensors of type (2,0).

The Calabi operator is defined by noting that by the K\"ahler symmetries the definition $$\langle C(Z_\alpha\otimes Z_\beta), \bar Z_\gamma\otimes \bar Z_\delta\rangle=R_{\alpha \bar \gamma \beta \bar \delta}$$
 induces a self-adjoint map $C \colon S^{2,0}\to S^{2,0}$. The Calabi curvature operator was introduced by Calabi-Vesentini in \cite{CV} to study the deformation theory of compact locally symmetric K\"ahler manifolds, and subsequently studied in \cite{BCV,OT,Sitaramayya}. 
 More recently, partial positivity of this operator has been used to establish vanishing of Hodge numbers and rigidity results \cite{BNPSWCalabiCurvatureOperator}, quadratic Weitzenb\"ock-Bochner-Kodaira formulas \cite{WangYangWeitzenbockBochnerKodaira}, and Laplacian comparison estimates \cite{FanXiongYangLaplacianComparison}. On complex projective space with the Fubini-Study metric, $C$ is a multiple of the identity operator.

Our main result is the following Tachibana-type theorem for the Calabi curvature operator.
\begin{maintheorem}
\label{thm main}\label{maintheorem}
  Let $(M,g)$ be a closed K\"ahler-Einstein manifold of complex dimension $n\geq 2$ with $\Ric(g)=\lambda g$ for some $\lambda \geq 0.$ If $C - \frac{\lambda}{n+1} \id$ is $\frac{2(n+1)}{5}$-nonnegative, then $(M,g)$ is isometric to complex projective space with the Fubini-Study metric or flat.
\end{maintheorem}

Here a self-adjoint operator with eigenvalues $\lambda_1 \leq \ldots \leq \lambda_m$ is said to be \textit{$k$-nonnegative} for some $k \geq 1$ if \begin{align*}
     \lambda_1 + \ldots + \lambda_{\lfloor k \rfloor} + (k - \lfloor k \rfloor) \lambda_{\lfloor k \rfloor +1} \ \geq \ 0.
 \end{align*}

Note that the average eigenvalue of $C$ is $\frac{2\lambda}{n+1}$ on a K\"ahler-Einstein manifold, so the operator $C - \frac{\lambda}{n+1} \id$ may be viewed as a ``half-trace''-operator sitting between $C$ and its traceless part. This operator arises naturally in the relevant Bochner formulas. 

Comparing the different nonnegativity conditions on algebraic Calabi curvature operators, not even full positivity of the spectrum of $C$ guarantees the condition in Theorem \ref{maintheorem}, and the condition of Theorem \ref{maintheorem} does not guarantee $k$-nonnegativity of $C$ for $k$ less than roughly $\frac{n+1}{5}$ either. Theorem \ref{maintheorem} shows that much stronger relations must hold for Calabi curvature operators realized by K\"ahler-Einstein manifolds. \vspace{2mm}

The proof of Theorem \ref{maintheorem} is based on a new Bochner formula for Calabi curvature operators. Bochner formulas for tensors are schematically of the type 
\begin{align*}\Delta|T|^2=c_1|\nabla T|^2+c_2\langle \Delta T, T\rangle + \text{curvature term}.
\end{align*}

For the curvature tensor of a K\"ahler-Einstein manifold, which is automatically harmonic, this looks like
\begin{align*}
 \Delta^\mathbb{C} |R_{\alpha \bar\beta \gamma \bar\delta}|^2 = |\nabla_\mu R_{\alpha \bar\beta \gamma \bar\delta}|^2 +|\nabla_{\bar\mu} R_{\alpha \bar\beta \gamma \bar\delta}|^2 + 
      \frac{1}{8}\Ric_L(R, \overline R),
      \end{align*}
where $\Delta^\C:=\frac12\left(\nabla_\mu\nabla_{\bar\mu}+ \nabla_{\bar\mu}\nabla_\mu\right)$ is the complex Laplacian and $\Ric_L$ is a naturally defined $0^{th}$ order curvature term. Our main technical result is an explicit formula for $\Ric_L$ in terms of the Calabi operator of $R$.

\begin{maintheorem}\label{thm main 2}\label{maintheorem2}
Let $(M^n,g)$ be a K\"ahler-Einstein manifold with $\Ric(g)=\lambda g$ for some $\lambda \in \R$, and let $\{S_A\}$ be a unitary eigenbasis of $S^{2,0}$ of $C$ with $C(S_A)=\sigma_A S_A$.
Then 
\begin{align*}
     \Ric_L(R, \overline R) =
      4 \sum \sigma_A |S_A R|^2 + 16 \tr(C^3) - 16\lambda \tr(C^2), 
\end{align*}
 where $S_A R$ is the natural action of $S^{2,0}$ on $(0,4)$-tensors, see  Definition \ref{DefAction}. 

\end{maintheorem}

This is established in Theorem \ref{prop:mainbochner}. When analyzing the sign of this term, one is naturally led to regroup in terms of $C-\frac{\lambda}{n+1}\id$. It would be interesting to know whether it can be controlled directly in terms of $C$ as well. However, it is not even apparent that the curvature term is nonnegative even if the whole spectrum of $C$ is. 

On the other hand, we note that nonnegative Calabi curvature operator implies nonnegative bisectional curvature. Hence, the classification result of Mok and Zhong \cite{MokZhongKEsymmetricSpaces} implies that for globally nonnegative Calabi curvature operators of closed K\"ahler-Einstein manifolds, the curvature term in the Bochner formula must vanish. Therefore, one is led to ask

\begin{question}
    What is the largest
$k(n)$ such that for every K\"ahler-Einstein algebraic curvature tensor $R$, $k$-nonnegativity of its
associated Calabi curvature operator \(C\) implies $\operatorname{Ric}_L(R,\overline R) \geq 0$?
\end{question}

In particular, one would hope at least for a positive resolution to

\begin{problem}
   Let \(R\) be a K\"ahler-Einstein algebraic curvature tensor. Prove that
if its associated Calabi curvature operator \(C\) is nonnegative, then $\operatorname{Ric}_L(R,\overline R) \geq 0$.
\end{problem}

We note that these difficulties are in clear contrast to similar situations with other curvature operators, where the Bochner formulas have easily apparent sign if the operator does. This includes the Riemannian operators of the first \cite{TachibanaOriginalTheorem} and second kind \cite{DaiFuEinsteinSecondKind, ChengWangEinsteinConeConditionsSecondKind}, as well as the standard K\"ahler curvature operator \cite{PetersenWinkVanishingHodge}.

These problems also do not arise when one is interested only in recovering the cohomology, see \cite{OT}. In fact, it is possible to show vanishing of primitive cohomology in terms of $k$-nonnegativity of $C$ itself for suitable $k$ depending on the degree, \cite{BNPSWCalabiCurvatureOperator, WangYangWeitzenbockBochnerKodaira}, but these methods cannot recover the isometry type.

\vspace{2mm}

\textbf{Acknowledgments.} This work was initiated at the workshop {\em The Bochner Technique} at the American Institute of Mathematics (AIM) in May 2026. We would like to thank AIM for this opportunity and for their hospitality during our stay.

XL's research is partially supported by NSF-DMS \#2553660 and a start-up grant at Auburn University. 

MW's research is partially supported by travel grant SFI-MPS-TSM-00026032 funded by the Simons Foundation International and administered by the Simons Foundation.

\section{Preliminaries}

\subsection{Conventions} Throughout the paper we use the Einstein summation convention unless indicated otherwise. Pairs of repeated lower indices consisting of a barred index and an unbarred index are summed from $1$ to $n$. Unless otherwise stated, all complex components are evaluated in a local unitary frame.

\textbf{Norms.}
When writing norms of the form $|T|^2$, we take this to mean the full tensor norm $\sum |T_{i_1, \ldots, i_k}|^2$, where each entry may have a bar or not. When writing norms in the form $|T_{a \bar b c \bar d}|^2$, we mean only that part of the norm with bars in the indicated positions. Thus, e.g., $|R|^2=4|R_{a \bar b c \bar d}|^2$ for K\"ahler curvature tensors.

\textbf{Laplacians.}
Different methods of defining the Laplacian of a function sometimes differ up to constants. We will use $\Delta^\C=\frac12\left(\nabla_\mu\nabla_{\bar\mu}
      + \nabla_{\bar\mu}\nabla_\mu\right)$, which differs by a factor of $\frac{1}{2}$ from the real trace of the real Hessian. Since we will mostly be doing complex computations, we write $\Delta f = \Delta^{\C}f=\frac{1}{2}\Delta^{\R}f$ and use the superscript only where explicit disambiguation seems preferable.

\subsection{K\"ahler Manifolds}
 For basic properties of K\"ahler manifolds, the reader is referred to \cite[Chapter 2]{Chow}. For a  K\"ahler manifold $(M,g)$ of complex dimension $n$, denote by $R_{\alpha\bar\beta\gamma\bar\delta}$ its K\"ahler curvature tensor.

Since $R_{\alpha \beta \cdot \cdot}=R_{\bar \alpha \bar \beta \cdot \cdot}=0,$ the first Bianchi identity implies the symmetries 
\begin{equation*}
R_{\alpha\bar\beta\gamma\bar\delta} = R_{\gamma\bar\beta\alpha\bar\delta}=R_{\alpha\bar\delta\gamma\bar\beta}=R_{\gamma\bar\delta\alpha\bar\beta}
\end{equation*}
and 
\begin{equation*}
    \overline{R_{\alpha\bar\beta\gamma\bar\delta}} =R_{\beta \bar \alpha \delta \bar \gamma}.
\end{equation*}

An algebraic K\"ahler curvature tensor $T$ is an element of $\otimes^4 T^*_{\C}M$ which is an algebraic curvature tensor, i.e., T satisfies
\begin{align*}
   T_{ijkl}& = -T_{jikl}=-T_{ijlk}=T_{klij}, \\ 
   & T_{ijkl}+T_{jkil}+T_{kijl}=0, \\
   & \hspace{7mm} \overline{T_{ijkl}} =T_{\bar i \bar j \bar k \bar l},
\end{align*}
that in addition satisfies the K\"ahler conditions 
\begin{align*}
    T_{\alpha \beta kl}=T_{\bar \alpha \bar \beta kl}=0.
\end{align*}

It follows that
\begin{equation*}
T_{\alpha\bar\beta\gamma\bar\delta} = T_{\gamma\bar\beta\alpha\bar\delta}=T_{\alpha\bar\delta\gamma\bar\beta}=T_{\gamma\bar\delta\alpha\bar\beta}.
\end{equation*}

For the curvature operator of a K\"ahler manifold, one also has the K\"ahler second Bianchi identity, which is given by
\begin{align*}
    \nabla_{\mu} R_{\alpha\bar\beta\gamma\bar\delta}  &= \nabla_{\alpha} R_{\mu \bar\beta\gamma\bar\delta} , \\ \nabla_{\bar\mu} R_{\alpha\bar\beta\gamma\bar\delta}  &= \nabla_{\bar \beta} R_{\alpha\bar\mu\gamma\bar\delta} \notag 
\end{align*}
and the commutator formula for a $(p,q)$-tensor $T$ is given by
\begin{align}
    \label{eq:2commutator}
\nabla_\alpha \nabla_{\bar\beta} & T_{\gamma_1 \cdots \gamma_p \bar\delta_1 \cdots \bar\delta_q} - \nabla_{\bar\beta} \nabla_\alpha T_{\gamma_1 \cdots \gamma_p \bar\delta_1 \cdots \bar\delta_q} \\
= \ & -\sum_{i=1}^p R_{\alpha \bar\beta \gamma_i \bar\eta} T_{\gamma_1 \cdots \gamma_{i-1}\eta \gamma_{i+1} \cdots \gamma_p \bar\delta_1 \cdots \bar\delta_q} \notag \\
\ & +\sum_{j=1}^q R_{\alpha \bar\beta \eta \bar{\delta}_j} T_{\gamma_1 \cdots \gamma_p \bar\delta_1 \cdots \bar{\delta}_{j-1} \bar\eta \bar{\delta}_{j+1} \cdots \bar\delta_q}. \notag
\end{align}

In complex coordinates, the K\"ahler curvature tensor decomposes as
\begin{eqnarray*}
    R_{\alpha\bar\beta\gamma\bar\delta} &=&  \frac{S}{n(n+1)}I_{\alpha\bar\beta\gamma\bar\delta} +B_{\alpha\bar\beta\gamma\bar\delta} \\
    && + \frac{1}{n+2} \left(g_{\alpha\bar\beta}\mathring{R}_{\gamma\bar\delta} +g_{\gamma\bar\delta}\mathring{R}_{\alpha\bar\beta} +g_{\alpha\bar\delta}\mathring{R}_{\gamma\bar\beta} +g_{\gamma\bar\beta}\mathring{R}_{\alpha\bar\delta}\right) ,
\end{eqnarray*}
where $\mathring{R}_{\alpha \bar \beta}=R_{\alpha \bar \beta} -\frac{S}{n}g_{\alpha \bar \beta}=R_{\alpha \bar \beta \gamma \bar \gamma} -\frac{S}{n}g_{\alpha \bar \beta} $ is the traceless Ricci tensor, $S = \Ric_{\alpha \bar \alpha}$ is half of the Riemannian scalar curvature and $B$ is the Bochner curvature tensor. 

\subsection{Curvature Operators}
\label{DefinitionCalabi}
Let \((M^n,g)\) be a K\"ahler manifold, and let
\(\{Z_\alpha\}_{\alpha=1}^n\) be a local unitary frame for
\(T^{1,0}M\). We extend \(g\) complex bilinearly to tensor products and
use the conventions
$$Z\wedge W=Z\otimes W-W\otimes Z$$
and 
$$Z\odot W=Z\otimes W+W\otimes Z. $$

The K\"ahler curvature tensor $R_{\alpha \bar \beta \gamma \bar \delta}$ induces two curvature operators. The K\"ahler curvature operator $K:\wedge^{1,1} \to \wedge^{1,1}$ is defined as
\begin{equation*}
    K(\omega)_{\alpha \bar\beta} = R_{\alpha \bar \beta \gamma \bar \delta} \omega_{\delta \bar \gamma } 
\end{equation*}

or 
\begin{equation*}
    \langle K(Z_\alpha \wedge \bar{Z}_\beta), \overline{Z_\gamma \wedge \bar{Z}_\delta} \rangle= 2R(Z_\alpha, \bar{Z}_\beta, Z_\delta, \bar{Z}_\gamma)=\langle R_{\alpha \bar \beta \delta \bar \gamma} Z_\gamma \wedge \bar Z_\delta, \overline{Z_\gamma \wedge \bar Z_\delta}\rangle.
\end{equation*}
The Calabi curvature operator $C:S^{2,0} \to S^{2,0}$ is given by
\begin{equation*}
    C(h)_{ \bar \beta  \bar \delta}=R_{\alpha \bar \beta \gamma \bar \delta} h_{ \bar \alpha  \bar \gamma}
\end{equation*}
or 
\begin{align*}
    \langle C & (Z_\alpha \od Z_\gamma) , \overline{Z_\beta \od Z_\delta}\rangle= 4R_{\alpha \bar \beta \gamma \bar \delta} , \\
    & C(Z_\alpha \od Z_\gamma)=R_{\alpha \bar \beta \gamma \bar \delta} Z_\beta \od Z_\delta
\end{align*}
where in the last line one has to note that the Einstein convention sums $\beta, \delta$, so for $\beta \ne \delta$ there are two terms with contribution $Z_\beta\od Z_\delta$, while for $\beta=\delta$ there is only one.

\begin{example}
The curvature tensor of $\mathbb{CP}^n$ with the Fubini-Study metric normalized with constant holomorphic sectional curvature $2$ is given by 
$$I_{\alpha\bar\beta\gamma\bar\delta}=g_{\alpha\bar\beta}g_{\gamma\bar\delta}+g_{\alpha\bar\delta}g_{\gamma\bar\beta}.$$
For $h\in S^{2,0}$, we have
$$C(h)_{\bar \beta \bar \delta}=I_{\alpha\bar\beta\gamma\bar\delta}h_{\bar \alpha \bar \gamma}=(g_{\alpha\bar\beta}g_{\gamma\bar\delta}+g_{\alpha\bar\delta}g_{\gamma\bar\beta})h_{\bar \alpha \bar \gamma}=h_{\bar \beta \bar \delta}+h_{\bar \delta \bar \beta}=2h_{\bar \beta \bar \delta}.$$
Therefore, $C=2 \id_{S^{2,0}}$ on $\mathbb{CP}^n$. 
\end{example}

\begin{example}
Let \(Q^n\) be the complex quadric with its standard metric, normalized with $\Ric_{\alpha\bar\beta}=n g_{\alpha\bar\beta}$. Its curvature tensor is given by
\[
R_{\alpha\bar\beta\gamma\bar\delta}
=
g_{\alpha\bar\beta}g_{\gamma\bar\delta}
+
g_{\alpha\bar\delta}g_{\gamma\bar\beta}
-
q_{\alpha\gamma}\overline{q_{\beta\delta}},
\]
where \(q_{\alpha\gamma}=q_{\gamma\alpha}\) is the symmetric bilinear form
associated to the real structure of the quadric. It satisfies $q_{\alpha\gamma}\overline{q_{\beta\delta}}g^{\gamma\bar\delta}
=
g_{\alpha\bar\beta}$.
In an adapted unitary frame, \(q_{\alpha\gamma}\) becomes the identity
matrix, but we keep the invariant notation \(q\) rather than writing
Kronecker symbols. The Calabi operator $C$ acts on \(h\in S^{2,0}\) by
$$Ch=2h - \langle h, \bar q \rangle q$$
Hence \(C\) has two distinct eigenvalues: $\sigma_0=2-n$ on the complex line spanned by \(q\), and $\sigma_1=2$ with multiplicity $\frac{(n-1)(n+2)}2$ on the orthogonal complement. 
\end{example}

\begin{example}
Let $M=\mathbb{CP}^p\times \mathbb{CP}^q$, $p+q=n$, be equipped with the product Kähler-Einstein metric normalized by $\operatorname{Ric}_{\alpha\bar\beta}= g_{\alpha\bar\beta}$.
The Calabi curvature operator diagonalizes according to the decomposition
\[
S^{2,0}(V\oplus W)
=
S^2V\oplus (V\odot W)\oplus S^2W,
\]
where $V=T^{1,0}\mathbb{CP}^p$ and $W=T^{1,0}\mathbb{CP}^q$. 
Its eigenvalues are $\frac{2}{p+1}$ on $S^2V$, $0$ on $V\odot W$, and $\frac{2}{q+1}$ on $S^2W$. 
\end{example}

\subsection{The Action}

Given any tensor $S\in T^{1,0}\otimes T^{1,0}$, we may view $S$ as a linear map $T_\C M\to T_\C M$ via the natural ``raising of the first index'' $$Z_i\otimes Z_k (v)=g(Z_i, v)Z_k.$$
Note that $S$ induces the zero map on $T^{1,0}$ and maps $T^{0,1}$ to $T^{1,0}$.

This lets us define an induced action on tensors:

\begin{definition}
\label{DefAction}
    Given $S\in T^{1,0} \otimes T^{1,0}$ and $T\in (T^{p,q})^* $ we may define $ST\in (T^{p-1, q+1})^* $ via 
    
    $ST(\cdot, \ldots, \cdot):= -T(S(\cdot), \cdot, \ldots, \cdot)-\ldots-T(\cdot, \ldots, \cdot, S(\cdot)).$
\end{definition}

The only case relevant for the rest of this paper is that where $S$ is symmetric and $T$ is an algebraic K\"ahler curvature operator. Here we have

\begin{lemma}\label{ST_computation}
    Let \(T\) be an algebraic K\"ahler curvature tensor,
and let \(S\in S^{2,0}\). 
Then for $x, y, z, w \in T^{1,0}$,
\begin{align*}
-ST(\bar x, \bar y, \bar z, w)=T(S(\bar x), \bar y, \bar z, w)-T(S(\bar y), \bar x, \bar z, w).
\end{align*}
\end{lemma}
\begin{proof}
    By definition we have $$-ST=T(S\cdot, \cdot,\cdot, \cdot)+T(\cdot, S\cdot,\cdot, \cdot)+T(\cdot, \cdot,S\cdot, \cdot)+T(\cdot, \cdot,\cdot, S\cdot).$$
    Thus $$-ST(\bar{Z}_\alpha, \bar{Z}_\beta, \bar{Z}_\gamma, Z_\delta)=T(S(\bar{Z}_\alpha), \bar{Z}_\beta, \bar{Z}_\gamma, Z_\delta)+T(\bar{Z}_\alpha, S(\bar{Z}_\beta), \bar{Z}_\gamma, Z_\delta)+0+0,
    $$
    or 
\begin{align*}
    -(ST)_{\bar \alpha \bar \beta \bar \gamma \delta}=S_{\bar \mu \bar \alpha}T_{\mu \bar \beta \bar \gamma \delta}+S_{\bar \mu \bar \beta}T_{\bar \alpha \mu \bar \gamma \delta}.
\end{align*}
The claim follows by noting that $T$ is skew-symmetric in the first two entries.
\end{proof}

In particular, we get for the norm:

\begin{lemma}\label{lemma |S_A B|^2}
Let \(T\) be an algebraic K\"ahler curvature tensor,
and let \(S\in S^{2,0}\). 
Then
\begin{equation*}
 |ST|^2=4 \sum\limits_{\alpha \beta \gamma \delta} \left | T(S(\bar{Z}_\alpha), \bar{Z}_\beta, \bar{Z}_\gamma, Z_\delta)-T(S(\bar{Z}_\beta), \bar{Z}_\alpha, \bar{Z}_\gamma, Z_\delta)\right |^2.
\end{equation*}
In particular, if \(S\) is diagonalized in a unitary frame \(\{Z_\alpha\}_{\alpha=1}^n\) of $T^{1,0}$ with nonnegative eigenvalues $\{\rho_\alpha\}_{\alpha=1}^n$, i.e., 
\[
    S=\sum_{\alpha=1}^n \rho_\alpha Z_\alpha \otimes Z_\alpha,
    \qquad \rho_\alpha\geq 0,
\]
then
\begin{equation}\label{NormST}
 |S T|^2 = 4\sum\limits_{\alpha \beta \gamma \delta}\left| \rho_\alpha T_{\alpha \bar \beta \gamma\bar\delta} - \rho_\beta T_{\beta \bar \alpha \gamma\bar\delta} \right|^2.    
\end{equation}
\end{lemma}

\begin{proof}
    The claim follows from Lemma \ref{ST_computation}, noting that the symmetries of $T$ pass to $ST$, so the four possible positions for the holomorphic entry contribute equally to the norm.
\end{proof}

\section{A New Bochner Formula}

The proof of the main theorem is based on a Bochner formula for curvature tensors. The standard Bochner technique produces expressions of the form
\begin{align}\label{realbochner}
    \Delta^\R \frac{1}{2} |T|^2 = |\nabla T|^2 + c \Ric_L(T,\bar T)
\end{align}
for tensors $T$ harmonic with respect to some Lichnerowicz-Laplacian $\Delta_L=\nabla^*\nabla+ c \Ric_L$, where $\Ric_L$ is a 0-th order curvature term depending on the type of tensor under consideration. For the curvature operator of an Einstein manifold, this is the case for $c=\frac{1}{2}$, see \cite[Theorem 9.4.2]{PetersenRiemGeom}. 

In complex notation, taking note of different Laplacians and norms, for algebraic K\"ahler curvature operators \eqref{realbochner} becomes 
\begin{align*}
    \Delta^\C |T_{\alpha \bar \beta \gamma \bar \delta}|^2 = |\nabla_\mu T_{\alpha \bar \beta \gamma \bar \delta}|^2 + |\nabla_{\bar \mu} T_{\alpha \bar \beta \gamma \bar \delta}|^2 + \frac{1}{8} \Ric_L(T,\bar T).
\end{align*}

It is clear from the form of the formula that a maximum principle forces $T$ to be parallel if we can show $\Ric_L(T, \bar T)\ge 0$. Thus the first step in the proof of the main theorem and the goal of this section is to produce an expression for $\Ric_L(R,\bar R)$ in terms of the Calabi curvature tensor.

\begin{theorem}\label{prop:mainbochner}
Let \((M,g)\) be a K\"ahler-Einstein manifold with \(R_{\alpha\bar\beta}=\lambda g_{\alpha\bar\beta}\). Let \(\{S_A\}\) be a unitary eigenbasis of \(S^{2,0}\) for the Calabi curvature operator \(C\), with \(C(S_A)=\sigma_A S_A\). Then 
\begin{eqnarray*}
    \Ric_L(R, \bar R)=4\sum_{A} \sigma_A |S_A R |^2 + 16 \tr(C^3) -16\lambda \tr(C^2).
\end{eqnarray*}

\end{theorem}

\begin{remark}\label{rem:traceremark}
    Expanding the curvature in the eigenvalue term into scalar and Bochner parts, we equivalently have, still assuming Einstein, that

$$\Ric_L(R, \bar R)=4\sum_{A} \sigma_A |S_A B |^2 + 16 \tr(C^3) -\frac{16(n+5)}{n+1}\lambda \tr(C^2) +\frac{32n(n+3)}{(n+1)^2} \lambda^3 .$$
\end{remark}

We begin by establishing a general formula for the Laplacian of K\"ahler curvature tensor fields. This is similar to the known special cases of the curvature tensor $R$, see \cite[Proposition 2.79]{Chow} or the Bochner tensor $B$, see \cite{IK04} or \cite{CDLR19}. 

\begin{proposition}\label{General_Second_Derviative}
Let $(M,g)$ be a K\"ahler manifold and let $T$ be a smooth algebraic K\"ahler curvature tensor field, i.e.,
$T$ is a section of $\otimes^4 T_{\mathbb{C}}^*M$ that is an algebraic K\"ahler operator at every point of $M$.

Suppose that $T$ also satisfies the K\"ahler second Bianchi identity
\begin{align}\label{eq:2K-Bianchi g}
    \nabla_{\mu} T_{\alpha\bar\beta\gamma\bar\delta}  &= \nabla_{\alpha} T_{\mu \bar\beta\gamma\bar\delta} , \\ \nabla_{\bar\mu} T_{\alpha\bar\beta\gamma\bar\delta}  &= \nabla_{\bar \beta} T_{\alpha\bar\mu\gamma\bar\delta}. \notag 
\end{align}
Then we have  
\begin{align*}
\Delta |T_{\alpha\bar\beta\gamma\bar\delta}|^2 
=& \ 2 \nabla_\gamma \nabla_{\bar \delta} \Ric(T)_{\alpha \beta}T_{\beta\bar\alpha\delta\bar\gamma} \\ 
& \ + 2 R_{\alpha\bar\mu\gamma\bar\nu}T_{\nu\bar\beta\mu\bar\delta}
     T_{\beta\bar\alpha\delta\bar\gamma}  
   - 4R_{\nu\bar\mu\gamma\bar\delta}T_{\alpha\bar\beta\mu\bar\nu}
     T_{\beta\bar\alpha\delta\bar\gamma}  \\
     & \ +2 R_{\nu \bar\beta} T_{\alpha\bar\nu\gamma\bar\delta}T_{\beta\bar\alpha\delta\bar\gamma}
   + |\nabla_\mu T_{\alpha\bar\beta\gamma\bar\delta}|^2
   + |\nabla_{\bar\mu}T_{\alpha\bar\beta\gamma\bar\delta}|^2.
\end{align*}
Here $\Ric(T)_{\alpha \bar\beta} = T_{\alpha \beta \mu \bar\mu}$. 
\end{proposition}

\begin{proof}
Using the second Bianchi identity \eqref{eq:2K-Bianchi g} and the commutator formula \eqref{eq:2commutator}, we calculate
\begin{eqnarray*}
\nabla_\gamma\nabla_{\bar\delta}\Ric(T)_{\alpha \bar \beta}
&=& \nabla_\gamma\nabla_{\bar\delta}T_{\alpha\bar\beta\mu\bar\mu} \\
&=&   \nabla_\gamma\nabla_{\bar\mu}T_{\alpha\bar\beta\mu\bar\delta} \\
&=& \nabla_{\bar\mu}\nabla_\gamma T_{\alpha\bar\beta\mu\bar\delta}
   - R_{\gamma\bar\mu\alpha\bar\nu}T_{\nu\bar\beta\mu\bar\delta}
   + R_{\gamma\bar\mu\nu\bar\beta}T_{\alpha\bar\nu\mu\bar\delta} \\
&&
   - R_{\gamma\bar\mu\mu\bar\nu}T_{\alpha\bar\beta\nu\bar\delta}
   + R_{\gamma\bar\mu\nu\bar\delta}T_{\alpha\bar\beta\mu\bar\nu} \\
&=& \nabla_{\bar\mu}\nabla_\mu T_{\alpha\bar\beta\gamma\bar\delta}
   - R_{\gamma\bar\mu\alpha\bar\nu}T_{\nu\bar\beta\mu\bar\delta}
   + R_{\gamma\bar\mu\nu\bar\beta}T_{\alpha\bar\nu\mu\bar\delta} \\
&&
   - R_{\gamma\bar\nu}T_{\alpha\bar\beta\nu\bar\delta}
   + R_{\nu\bar\mu\gamma\bar\delta}T_{\alpha\bar\beta\mu\bar\nu}.
\end{eqnarray*}
Also,
\begin{eqnarray*}
    \nabla_\mu\nabla_{\bar\mu}T_{\alpha\bar\beta\gamma\bar\delta}
&=& \nabla_{\bar\mu}\nabla_\mu T_{\alpha\bar\beta\gamma\bar\delta}
- R_{\mu\bar\mu\alpha\bar\nu}T_{\nu\bar\beta\gamma\bar\delta}
   + R_{\mu\bar\mu\nu\bar\beta}T_{\alpha\bar\nu\gamma\bar\delta} \\
   &&    - R_{\mu\bar\mu\gamma\bar\nu}T_{\alpha\bar\beta\nu\bar\delta}
   + R_{\mu\bar\mu\nu\bar\delta}T_{\alpha\bar\beta\gamma\bar\nu} \\
&=& \nabla_{\bar\mu}\nabla_\mu T_{\alpha\bar\beta\gamma\bar\delta} 
 - R_{\alpha\bar\nu}T_{\nu\bar\beta\gamma\bar\delta}
   + R_{\nu\bar\beta}T_{\alpha\bar\nu\gamma\bar\delta} \\
   &&  - R_{\gamma\bar\nu}T_{\alpha\bar\beta\nu\bar\delta}
   + R_{\nu\bar\delta}T_{\alpha\bar\beta\gamma\bar\nu}.
\end{eqnarray*}
Combining the formulas above yields
\begin{eqnarray*}
&&  \Delta T_{\alpha\bar\beta\gamma\bar\delta} := \frac12\left(\nabla_\mu\nabla_{\bar\mu}
      + \nabla_{\bar\mu}\nabla_\mu\right)
      T_{\alpha\bar\beta\gamma\bar\delta} \\
      &=& \nabla_{\bar\mu}\nabla_\mu T_{\alpha\bar\beta\gamma\bar\delta}
   - \frac12\left(
       R_{\alpha\bar\nu}T_{\nu\bar\beta\gamma\bar\delta}
       - R_{\nu\bar\beta}T_{\alpha\bar\nu\gamma\bar\delta} + R_{\gamma\bar\nu}T_{\alpha\bar\beta\nu\bar\delta}
       - R_{\nu\bar\delta}T_{\alpha\bar\beta\gamma\bar\nu}
     \right) \\
&=& \nabla_\gamma\nabla_{\bar\delta}\Ric(T)_{\alpha\bar\beta}
   + R_{\gamma\bar\mu\alpha\bar\nu}T_{\nu\bar\beta\mu\bar\delta}
   - R_{\gamma\bar\mu\nu\bar\beta}T_{\alpha\bar\nu\mu\bar\delta} 
   - R_{\nu\bar\mu\gamma\bar\delta}T_{\alpha\bar\beta\mu\bar\nu} \\
&& - \frac12\left(
       R_{\alpha\bar\nu}T_{\nu\bar\beta\gamma\bar\delta}
       - R_{\nu\bar\beta}T_{\alpha\bar\nu\gamma\bar\delta}
       - R_{\gamma\bar\nu}T_{\alpha\bar\beta\nu\bar\delta}
       - R_{\nu\bar\delta}T_{\alpha\bar\beta\gamma\bar\nu}
     \right),
\end{eqnarray*}
which holds on any K\"ahler manifold.

We then calculate
\begin{eqnarray*}
 \Delta |T_{\alpha\bar\beta\gamma\bar\delta}|^2 
&=& \nabla_\mu \nabla_{\bar{\mu}} (T_{\alpha\bar\beta\gamma\bar\delta} T_{\beta\bar\alpha\delta\bar\gamma} ) \\
&=& 2(\Delta T_{\alpha\bar\beta\gamma\bar\delta})T_{\beta\bar\alpha\delta\bar\gamma}
   + |\nabla_\mu T_{\alpha\bar\beta\gamma\bar\delta}|^2
   + |\nabla_{\bar\mu}T_{\alpha\bar\beta\gamma\bar\delta}|^2 \\
&=& 2 \nabla_\gamma\nabla_{\bar\delta}\Ric(T)_{\alpha\bar\beta} T_{\beta\bar\alpha\delta\bar\gamma} +2 R_{\gamma\bar\mu\alpha\bar\nu}T_{\nu\bar\beta\mu\bar\delta}
     T_{\beta\bar\alpha\delta\bar\gamma} 
- 2 R_{\gamma\bar\mu\nu\bar\beta}T_{\alpha\bar\nu\mu\bar\delta}
     T_{\beta\bar\alpha\delta\bar\gamma} \\
     && 
   - 2 R_{\nu\bar\mu\gamma\bar\delta}T_{\alpha\bar\beta\mu\bar\nu}
     T_{\beta\bar\alpha\delta\bar\gamma} \\
     && - \left(
       R_{\alpha\bar\nu}T_{\nu\bar\beta\gamma\bar\delta}
       - R_{\nu\bar\beta}T_{\alpha\bar\nu\gamma\bar\delta}
       - R_{\gamma\bar\nu}T_{\alpha\bar\beta\nu\bar\delta}
       - R_{\nu\bar\delta}T_{\alpha\bar\beta\gamma\bar\nu}
     \right)T_{\beta\bar\alpha\delta\bar\gamma}  \\
     &&  + |\nabla_\mu T_{\alpha\bar\beta\gamma\bar\delta}|^2
   + |\nabla_{\bar\mu}T_{\alpha\bar\beta\gamma\bar\delta}|^2 \\ 
&=& 2 \nabla_\gamma\nabla_{\bar\delta}\Ric(T)_{\alpha\bar\beta} T_{\beta\bar\alpha\delta\bar\gamma} + 2 R_{\gamma\bar\mu\alpha\bar\nu}T_{\nu\bar\beta\mu\bar\delta}
     T_{\beta\bar\alpha\delta\bar\gamma}
- 4 R_{\gamma\bar\mu\nu\bar\beta}T_{\alpha\bar\nu\mu\bar\delta}
     T_{\beta\bar\alpha\delta\bar\gamma}  \\
     && +2 R_{\nu \bar\beta} T_{\alpha\bar\nu\gamma\bar\delta}T_{\beta\bar\alpha\delta\bar\gamma}  + |\nabla_\mu T_{\alpha\bar\beta\gamma\bar\delta}|^2
   + |\nabla_{\bar\mu}T_{\alpha\bar\beta\gamma\bar\delta}|^2. 
\end{eqnarray*}
\end{proof}

We now relate the curvature terms to the Calabi curvature operator. 

\begin{proposition}\label{prop_Eigenvalue_Term}
Suppose that $\{S_A\}_{A=1}^{N},$ $N=\frac{n(n+1)}{2},$ is a unitary basis for $S^{2,0}$, and $R$ is induced from the algebraic Calabi operator $C$ defined by
\[
C(S_A)=\sigma_A S_A.
\]
Then for any algebraic Kähler curvature tensor $T$,  
\begin{align*}
\sum_{A=1}^{N}\sigma_A|S_AT|^2
= 8 R_{\nu \bar\beta} T_{\alpha\bar\nu\gamma\bar\delta}T_{\beta\bar\alpha\delta\bar\gamma}
- 8 R_{\nu\bar\mu\gamma\bar\delta}T_{\alpha\bar\beta\mu\bar\nu}
     T_{\beta\bar\alpha\delta\bar\gamma}.
\end{align*}
\end{proposition}

\begin{proof}
Fix a unitary frame
$\{Z_\alpha\}_{\alpha=1}^{n}$ of $T^{1,0}M$, and identify
$S\in S^{2,0}$ with the map
\[
S:T^{0,1}M\longrightarrow T^{1,0}M
\]
obtained by raising an index. Write
\[
S_A(\overline Z_\alpha)=s^A_{\bar \mu \bar \alpha}Z_\mu,
\qquad
s^A_{\bar \mu \bar \alpha}=s^A_{\bar \alpha \bar \mu}.
\]

Since by definition $C_R=\sum\sigma_A S_A\otimes \bar S_A$, we have

\begin{equation*}
\sum_A
\sigma_A
s^A_{ \bar \mu \bar \alpha}
\overline{s^A_{ \bar \nu \bar \beta}}
=
R_{\nu\bar\mu\beta\bar\alpha}
\end{equation*}

By Lemma \ref{lemma |S_A B|^2}, applied to $T$, we have
\begin{equation*}
|S_AT|^2
=
4 | 
s^A_{\bar \mu \bar \alpha}T_{\mu\bar\beta\gamma\bar\delta}
-
s^A_{\bar \mu \bar \beta}T_{\mu\bar\alpha\gamma\bar\delta}
|^2
\end{equation*}
and hence
\begin{align*}
\sum_A\sigma_A|S_AT|^2
={}&
4\sum_A\sigma_A
\Bigl[
s^A_{ \bar \mu \bar \alpha}\overline{s^A_{ \bar \nu \bar \alpha}}\,
T_{\mu\bar\beta\gamma\bar\delta}
\overline{T_{\nu\bar\beta\gamma\bar\delta}}
-
s^A_{ \bar \mu \bar \alpha}\overline{s^A_{ \bar \nu \bar \beta}}\,
T_{\mu\bar\beta\gamma\bar\delta}
\overline{T_{\nu\bar\alpha\gamma\bar\delta}}
\\
&\hspace{13mm}
-
s^A_{ \bar \mu \bar \beta}\overline{s^A_{ \bar \nu \bar \alpha}}\,
T_{\mu\bar\alpha\gamma\bar\delta}
\overline{T_{\nu\bar\beta\gamma\bar\delta}}
+
s^A_{ \bar \mu \bar \beta}\overline{s^A_{ \bar \nu \bar \beta}}\,
T_{\mu\bar\alpha\gamma\bar\delta}
\overline{T_{\nu\bar\alpha\gamma\bar\delta}}
\Bigr] \\
= {} & 8\sum_A\sigma_A
s^A_{ \bar \mu \bar \alpha}\overline{s^A_{ \bar \nu \bar \alpha}}\,
T_{\mu\bar\beta\gamma\bar\delta}
\overline{T_{\nu\bar\beta\gamma\bar\delta}}
-
8\sum_A\sigma_A
s^A_{ \bar \mu \bar \alpha}\overline{s^A_{ \bar \nu \bar \beta}}\,
T_{\mu\bar\beta\gamma\bar\delta}
\overline{T_{\nu\bar\alpha\gamma\bar\delta}}
\\
= {} & 8  R_{\nu\bar\mu\alpha\bar\alpha}
T_{\mu\bar\beta\gamma\bar\delta}
\overline{T_{\nu\bar\beta\gamma\bar\delta}} - 8 R_{\nu \bar \mu \beta \bar \alpha} T_{\mu\bar\beta\gamma\bar\delta}
\overline{T_{\nu\bar\alpha\gamma\bar\delta}} \\
= {} & 8  R_{\nu\bar\mu}
T_{\mu\bar\beta\gamma\bar\delta}
\overline{T_{\nu\bar\beta\gamma\bar\delta}} - 8 R_{\nu\bar\mu\beta\bar\alpha}
T_{\mu\bar\beta\gamma\bar\delta}
\overline{T_{\nu\bar\alpha\gamma\bar\delta}},
\end{align*}
which yields the claim after using the K\"ahler symmetries of $T$ and relabeling the indices.
\end{proof}

Therefore,
\begin{proposition}\label{General_Bochner_Formula} Every smooth (not necessarily harmonic) K\"ahler curvature tensor field $T$ satisfying the K\"ahler second Bianchi identity satisfies
\begin{align*}
\frac{1}{8}\Ric_L(T, \bar T)=& \ \Delta |T_{\alpha\bar\beta\gamma\bar\delta}|^2 - 2 \nabla_\gamma \nabla_{\bar \delta} \Ric(T)_{\alpha \bar \beta}T_{\beta\bar\alpha\delta\bar\gamma} -|\nabla_\mu T_{\alpha\bar\beta\gamma\bar\delta}|^2
   - |\nabla_{\bar\mu}T_{\alpha\bar\beta\gamma\bar\delta}|^2 \\
= & \ 2 R_{\alpha\bar\mu\gamma\bar\nu}T_{\nu\bar\beta\mu\bar\delta}
     T_{\beta\bar\alpha\delta\bar\gamma}  
   - 4R_{\nu\bar\mu\gamma\bar\delta}T_{\alpha\bar\beta\mu\bar\nu}
     T_{\beta\bar\alpha\delta\bar\gamma}  +2 R_{\nu \bar\beta} T_{\alpha\bar\nu\gamma\bar\delta}T_{\beta\bar\alpha\delta\bar\gamma}  \\
   = & \ 2 \tr (C_R \circ C_T \circ C_T) +\frac12 \sum_A \sigma_A |S_A T|^2 -2 R_{\nu \bar\beta} T_{\alpha\bar\nu\gamma\bar\delta}T_{\beta\bar\alpha\delta\bar\gamma}. 
\end{align*}  
Here $C_T$ denotes the Calabi curvature operator induced by $T$. 
\end{proposition}

\begin{proof}
The second equality is the statement of Proposition \ref{General_Second_Derviative}, and the third holds by Proposition \ref{prop_Eigenvalue_Term}. The divergence terms in the first line are the necessary adjustments to the definition of $\Ric_L$ when $T$ is not assumed to be harmonic. Note that we have seen the right side is a $0$-th order curvature term and that it agrees with the definition given earlier for harmonic $T$, since in that case the divergence terms vanish.
\end{proof}

As corollary, we get

\begin{proof}[Proof of Theorem \ref{prop:mainbochner}]
    Apply Proposition \ref{General_Bochner_Formula} to $T=R$ and use $Ric=\lambda g$.
\end{proof}

\section{Proof of the Main Theorem}

To apply the weight principle of \cite[Section 3]{NPWBettiNumbersAndCOSK}, we need to find the total weight $\sum_{A} |S_A B |^2$ and estimate the highest weight $|S_A B|^2$ from above.

\begin{proposition}\label{prop 4.1 total weight}
Let $T$ be an algebraic K\"ahler curvature tensor and let
\(\{S_A\}_{A=1}^N\) be a unitary basis of \(S^{2,0}\). Then
\begin{equation*}
\sum_{A} |S_A T |^2    =  4 n |T_{\alpha \bar \beta \gamma \bar \delta}|^2 - 4 |T_{\alpha \bar \beta}|^2.
\end{equation*}
\end{proposition}
\begin{proof}
We apply Proposition \ref{prop_Eigenvalue_Term} with $R=I$. In this case, we have $C_R=2\operatorname{id}_{S^{2,0}}$, $\sigma_A=2$ for all $A$, and  $R_{\alpha \bar \beta}=(n+1)g_{\alpha \bar \beta}$. Hence, 
\begin{align*}
     I_{\xi \bar\mu \eta \bar \nu} T_{\nu \bar \eta \gamma \bar \delta} T_{\mu \bar \xi \delta \bar \gamma} = (g_{\xi \bar \mu} g_{\eta \bar \nu} +g_{\xi \bar \nu}g_{\eta \bar \mu})T_{\nu \bar \eta \gamma \bar \delta} T_{\mu \bar \xi \delta \bar \gamma} =|T_{\xi \bar \xi \gamma \bar \delta}|^2 + |T_{\alpha \bar \beta \gamma \bar \delta}|^2
\end{align*} 
 and Proposition \ref{prop_Eigenvalue_Term} shows
\begin{align*}
  \sum_A |S_A T|^2=4(n+1)|T_{\alpha \bar \beta \gamma \bar \delta}|^2-4|T_{\alpha \bar \beta \gamma \bar \delta}|^2- 4 |T_{\gamma \bar \delta}|^2=4n|T_{\alpha \bar \beta \gamma \bar \delta}|^2 - 4 |T_{\gamma \bar \delta}|^2.  
\end{align*}
\end{proof}

Since the terms in Theorem \ref{prop:mainbochner} do not individually vanish on the identity component, it will be convenient for the proof of the main theorem to pass to the point of view of the traceless part only. 

Let $C_B=C-\frac{2\lambda}{n+1} \id_{S^{2,0}}$ be the traceless part of the Calabi curvature operator (or equivalently, since $C$ is Einstein, Calabi operator induced by the Bochner tensor $B$). The eigenvalues of $C_B$ are given by 
\begin{equation*}
    \tau_A =\sigma_A -\frac{2\lambda}{n+1},
\end{equation*} 
where $\sigma_A$ are the corresponding eigenvalues of $C$. We have 
\begin{equation*}
    \sum_A \tau_A =0
\end{equation*}
and 
\begin{equation*}
    \sum_A \tau_A^2 = |B_{\alpha \bar \beta \gamma\bar\delta}|^2 = \frac{1}{4}|B|^2.
\end{equation*}

We can then rewrite the description of $\Ric_L$ given in Remark \ref{rem:traceremark} as
\begin{align}
\frac{1}{8}\Ric_L(R, \bar R) = & \ \frac{1}{2}\sum_{A} \left(\tau_A +\frac{2\lambda}{n+1}\right) |S_A B |^2 + 2 \sum_A \left(\tau_A +\frac{2\lambda}{n+1}\right)^3 \notag \\
& \ -\frac{2(n+5)}{n+1}\lambda  \sum_A \left(\tau_A +\frac{2\lambda}{n+1}\right)^2+\frac{4n(n+3)}{(n+1)^2} \lambda^3 \notag \\
=& \ \frac{1}{2}\sum_{A} \tau_A  |S_A B |^2 + 2 \sum_A \tau_A^3  + 2 \lambda \sum_A \tau_A^2. \label{RicLinTraceFreeEigenvalues}
\end{align}

This immediately gives us a weak version of Theorem \ref{maintheorem}:

\begin{proposition}
If $C - \frac{\lambda}{n+1} \id_{S^{2,0}}$ is nonnegative, then $\Ric_L(R, \bar R)\geq 0$. 
\end{proposition}
\begin{proof}
Denote by $\tau_{\min}=\min\{\tau_A\}$ the smallest eigenvalue of $C_B$. Then \eqref{RicLinTraceFreeEigenvalues} implies
\begin{align*}
\frac{1}{8}\Ric_L(R, \bar R) 
\geq \ & \frac{1}{2} \tau_{\min} \sum_{A} |S_A B |^2 + 2 \tau_{\min} \sum_A \tau_A^2 +2 \lambda \sum_A \tau_A^2 \\
= \ & 2\left( (n +1)\tau_{\min}  +\lambda \right) |B_{\alpha \bar \beta \gamma \bar \delta}|^2 \geq 0. 
\end{align*}
\end{proof}

The main work will be in relaxing the strict requirement of nonnegativity. The framework for this was developed in \cite{NPWBettiNumbersAndCOSK}. Explicitly, we are reduced to the following optimization problem:

\begin{lemma}\label{maximizationproblem}
   Set 
   \begin{align*}
       c=\max\limits_{|S|=1}\frac{|SB|^2+4|B(S)|^2}{|B_{}|^2}.
   \end{align*}
    If $C-\frac{\lambda}{n+1}$ is $\frac{(n+1)}{c}$-nonnegative, then $\Ric_L(R, \bar R)\ge 0.$
\end{lemma}
\begin{proof} By \eqref{RicLinTraceFreeEigenvalues} we have 
\begin{align*}
\frac{1}{4}\Ric_L(R, \bar R) =& \ \sum_{A} (\sigma_A  - \frac{2\lambda}{n+1})|S_A B |^2 + 4 \sum_A (\sigma_A  - \frac{2\lambda}{n+1})^3 + 4 \lambda \sum_A (\sigma_A  - \frac{2\lambda}{n+1})^2 \\
= & \ \sum_{A} (\sigma_A  - \frac{\lambda}{n+1})|S_A B |^2 + 4 \sum_A (\sigma_A  - \frac{\lambda}{n+1})(\sigma_A  - \frac{2\lambda}{n+1})^2\\
=& \ \sum_{A} (\sigma_A  - \frac{\lambda}{n+1})(|S_A B |^2 +4(\sigma_A  - \frac{2\lambda}{n+1})^2). 
\end{align*}
This is a weighted sum of eigenvalues of $C-\frac{\lambda}{n+1}$. The total weight is
\begin{align*}
\Omega=\sum_{A}(|S_A B |^2 +4(\sigma_A  - \frac{2\lambda}{n+1})^2)= (n+1)|B|^2
\end{align*}
and the highest weight is
\begin{align*}
 \omega=\max\limits_A|S_A B |^2 +4(\sigma_A  - \frac{2\lambda}{n+1})^2 = \max\limits_A|S_A B |^2 +4|B(S_A)|^2. 
\end{align*}
By the weight principle \cite[Theorem 3.6]{NPWBettiNumbersAndCOSK}, $\Ric_L(R, \bar R)\ge 0$ if $C-\frac{\lambda}{n+1}$ is $\frac{\Omega}{\omega}$-nonnegative, but $\frac{\Omega}{\omega}\ge \frac{n+1}{c}$ by definition of $c$, so this is true by assumption.
\end{proof}

This leaves us with having to estimate the highest weight. The straightforward bound is the following:
 
\begin{lemma}\label{prop:8bound}
    $|S B |^2 \leq 8 |S|^2 |B_{\alpha \bar \beta \gamma \bar \delta}|^2.$
\end{lemma}

\begin{proof}
Using the Cauchy-Schwarz inequality
$$|ax-by|^2 \leq (a^2+b^2)(|x|^2+|y|^2),$$
we get from \eqref{NormST}
\begin{align*}
 |S B|^2
 =& \ 4 \sum_{\alpha \neq \beta} \left| \rho_\alpha B_{\alpha \bar \beta \gamma\bar\delta} - \rho_\beta B_{\beta \bar \alpha \gamma\bar\delta} \right|^2 \\
 \leq & \ 4\sum_{\alpha \neq \beta}  (\rho_\alpha^2 +\rho_\beta^2) (|B_{\alpha \bar \beta \gamma\bar\delta}|^2 + |B_{\beta \bar \alpha \gamma\bar\delta}|^2) \\
 \leq & \ 8  |S|^2 |B_{\alpha \bar \beta \gamma\bar\delta}|^2, 
\end{align*}
noting for the first line that the term in the sum vanishes identically for $\alpha = \beta$.
\end{proof}

If there is to be any improvement over this obvious bound, one has to hope that some cancellation happens based on the specific form of the action. In fact, with some effort, it is possible to prove an (asymptotically) sharp estimate for the action term:

\begin{lemma}\label{6bound}
    $|SB|^2\le 6 |S|^2|B_{\alpha \bar \beta \gamma \bar \delta}|^2$.
\end{lemma}
\begin{proof}
Fix $S$. In a basis diagonalizing $S$, we have \begin{align*}
    |SB|^2=4\sum\limits_{\alpha \beta \gamma \delta}|\rho_\alpha B_{\alpha \bar \beta \gamma \bar \delta}-\rho_\beta B_{\beta \bar \alpha \gamma \bar \delta}|^2=4 \sum\limits_I \sum\limits_{\{\alpha, \beta, \gamma, \delta\}=I}|\rho_\alpha B_{\alpha \bar \beta \gamma \bar \delta}-\rho_\beta B_{\beta \bar \alpha \gamma \bar \delta}|^2,
\end{align*}
where $I$ runs over all multisets of indices of size $4$. Since different summands do not see the same coefficients of curvature, the overall ratio $|SB|/|B|$ is bounded above by the maximum ratio for each summand. This leaves us with 5 cases to consider: $I=\{\alpha,\alpha,\alpha,\alpha\},\{\alpha,\alpha,\alpha,\beta\},\{\alpha,\alpha,\beta, \beta\},\{\alpha,\alpha,\beta, \gamma\},\{\alpha,\beta, \gamma, \delta\},$ where differently named variables are presumed to be different. 

Throughout we write $|B_I|^2$ to be the sum over $|B_{\alpha \bar \beta \gamma \bar \delta}|^2$ with indices according to $I$, and conjugations only in the fixed positions. As the index set is fixed in advance, terms without explicit sums are to be read without Einstein convention in mind.\vspace{2mm}

$I=\{\alpha,\alpha,\alpha,\alpha\}$: \\
Here $\sum\limits_{\{\alpha, \beta, \gamma, \delta\}=I}|\rho_\alpha B_{\alpha \bar \beta \gamma \bar \delta}-\rho_\beta B_{\beta \bar \alpha \gamma \bar \delta}|^2$ only has one term, which is zero. \vspace{2mm}

$I=\{\alpha,\alpha,\alpha,\beta\}$: \\
Here 
\begin{align*}
    4\sum\limits_{\{\alpha, \beta, \gamma, \delta\}=I} & |\rho_\alpha B_{\alpha\bar \beta \gamma \bar \delta}-\rho_\beta B_{\beta \bar \alpha \gamma \bar \delta}|^2  \\
     = & \ 8 |\rho_\beta B_{\beta \bar \alpha \alpha \bar \alpha}-\rho_\alpha B_{\alpha \bar \beta \alpha \bar \alpha}|^2 \\ 
     \leq & \  \ 8 (\rho_\beta + \rho_\alpha)^2|B_{\alpha \bar \beta \alpha \bar \alpha}|^2 \\
     \leq & \  4 |S|^2 |B_I|^2.
\end{align*}

$I=\{\alpha,\alpha,\beta, \beta\}$: \\
Here \begin{align*}
        4\sum\limits_{\{\alpha, \beta, \gamma, \delta\}=I} & |\rho_\alpha B_{\alpha \bar \beta \gamma \bar \delta}-\rho_\beta B_{\beta \bar \alpha \gamma \bar \delta}|^2\\
        =& \ 8 (|\rho_\alpha B_{\alpha \bar \beta \alpha \bar \beta}-\rho_\beta B_{\beta \bar \alpha \alpha \bar \beta}|^2+|\rho_\alpha B_{\alpha \bar \beta \beta \bar \alpha}-\rho_\beta B_{\beta \bar \alpha \beta \bar \alpha}|^2)\\
        \leq & \ 8(\rho_\alpha^2+\rho_\beta^2)(3|B_{\alpha \bar \alpha \beta \bar \beta}|^2+\frac{3}{2}|B_{\alpha \bar \beta \alpha \bar \beta}|^2)\\
        \leq & \ 6|S|^2(4|B_{\alpha \bar \alpha \beta \bar \beta}|^2+2|B_{\alpha \bar \beta \alpha  \bar \beta}|^2)\\
        =& \ 6|S|^2|B_I|^2,
\end{align*}
where in the third line we used that for any real $a,b$ and complex $x,y$ one has by Cauchy-Schwarz that 
\begin{align*}
    |ax-by|^2+|a\bar y-b\bar x|^2&\le(a^2+b^2)(|x|^2+|y|^2)+4|a||b||x||y|\\
    &\le (a^2+b^2)(|x|^2+|y|^2+2|x|^2+1/2|y|^2)\\
    &=(a^2+b^2)(3|x|^2+\frac{3}{2}|y|^2).
\end{align*}

$I=\{\alpha,\alpha,\beta, \gamma\}$: \\
Here \begin{align*}
    4 & \sum \limits_{\{\alpha, \beta, \gamma, \delta\}=I}  |\rho_\alpha B_{\alpha\bar \beta \gamma \bar \delta}-\rho_\beta B_{\beta \bar \alpha \gamma \bar \delta}|^2\\
     = & \ 8 (|\rho_\alpha B_{\alpha \bar \beta \alpha \bar \gamma}-\rho_\beta B_{\beta \bar \alpha \alpha \bar \gamma}|^2+|\rho_\alpha B_{\alpha \bar \beta \gamma \bar \alpha}-\rho_\beta B_{\beta \bar \alpha \gamma \bar \alpha}|^2+|\rho_\alpha B_{\alpha \bar \gamma \alpha \bar \beta}-\rho_\gamma B_{\gamma \bar \alpha \alpha \bar \beta}|^2\\
    & \hspace{4mm} +|\rho_\alpha B_{\alpha \bar \gamma \beta \bar \alpha}-\rho_\gamma B_{\gamma \bar \alpha \beta \bar \alpha}|^2+|\rho_\beta B_{\beta \bar \gamma \alpha \bar \alpha}-\rho_\gamma B_{\gamma \bar \beta \alpha \bar \alpha}|^2)\\
    \leq & \ 8\big((\frac{3}{2}\rho_\alpha^2+3\rho_\beta^2)(|B_{\alpha \bar \beta \alpha \bar \gamma}|^2+|B_{\beta \bar \alpha \alpha \bar \gamma}|^2)+(\frac{3}{2}\rho_\alpha^2+3\rho_\gamma^2)(|B_{\alpha \bar \gamma \alpha \bar \beta}|^2+|B_{\gamma \bar \alpha \alpha \bar \beta}|^2)\\
    & \hspace{4mm} +(\rho_\beta+\rho_\gamma)^2|B_{\beta \bar \gamma \alpha \bar \alpha}|^2)\big)\\
    = & \ 24(\rho_\alpha^2+\rho_\beta^2+\rho_\gamma^2)(|B_{\alpha\bar \beta \alpha \bar \gamma}|^2+|B_{\beta \bar \alpha \alpha \bar \gamma}|^2)+8(\rho_\beta+\rho_\gamma)^2|B_{\beta \bar \gamma \alpha \bar \alpha}|^2\\
    \leq & \ 6|S|^2(4|B_{\alpha \bar \beta \alpha \bar \gamma}|^2+8|B_{\beta \bar \alpha \alpha \bar \gamma}|^2)\\
    = & \ 6|S|^2|B_I|^2,
\end{align*}
where once more we used Cauchy-Schwarz to see
\begin{align*}
    |ax-by|^2+|a\bar y-b\bar x|^2 
    \le& (a^2+b^2)(|x|^2+|y|^2)+4|a||b||x||y|\\
    \le& (a^2+b^2)(|x|^2+|y|^2)+ 2|a||b|(|x|^2+|y|^2)\\
    \le&(|x|^2+|y|^2)(a^2+b^2+\frac{1}{2}a^2+2b^2)\\
    =&(\frac{3}{2}a^2+3b^2)(|x|^2+|y|^2)
\end{align*}
for arbitrary real $a,b$ and complex $x,y$.

Note here that the estimate $24+8\le 6\cdot 8$ is not sharp, and using the same method one can obtain a somewhat better constant in this case. \vspace{2mm}

$I=\{\alpha,\beta, \gamma, \delta\}$:\\
Here \begin{align*}
    4\sum\limits_{\{\alpha, \beta, \gamma, \delta\}=I}  &|\rho_\alpha B_{\alpha\bar \beta \gamma \bar \delta}-\rho_\beta B_{\beta \bar \alpha \gamma \bar \delta}|^2\\
    \leq & \ 4\sum\limits_{\{\alpha,\beta,\gamma,\delta\}=I}(\rho_\alpha^2+\rho_\beta^2)(|B_{\alpha\bar \beta \gamma \bar \delta}|^2+|B_{\beta \bar \alpha \gamma \bar \delta}|^2)\\
    = & \ 2\sum\limits_{\{\alpha,\beta,\gamma,\delta\}=I}(\rho_\alpha^2+\rho_\beta^2+\rho_\gamma^2+\rho_\delta^2)(|B_{\alpha \bar \beta \gamma \bar \delta}|^2+|B_{\beta \bar \alpha \gamma \bar \delta}|^2)\\
    \leq & \ 4|S|^2|B_I|^2.
\end{align*}

Thus we see $|SB|^2\le 6|S|^2|B_{\alpha \bar \beta \gamma \bar \delta}|^2$ in all cases.
\end{proof}

\begin{remark}\label{rem:sharpness}
    The estimate is sharp on algebraic K\"ahler curvature operators not subject to any additional conditions. Define $T$ so that $T_{1\bar 2 1 \bar 2}=\pm2T_{1\bar 1 2 \bar 2}\ne 0$ are real and such that all coefficients not determined by these vanish, and take $S=Z_1\otimes Z_1+Z_2\otimes Z_2$.

    To note that it is possible to choose $T$ like that, note that algebraic curvature operators correspond to self-adjoint operators on $S^{2,0}$ via the Calabi operator, and that $T$ is induced by $$C=\begin{pmatrix}
        0 & a & 0 \\ a & 0 & 0 \\ 0 & 0 & b
    \end{pmatrix}$$ in the basis $Z_1\odot Z_1, Z_2\odot Z_2, Z_1 \odot Z_2$ for suitable $a,b>0$. 

    This algebraic example is Einstein and $S$ is an eigenvector of $C$, so neither of those conditions can give general improvement. A natural extension to an algebraic Bochner tensor in dimension $n$ gives examples with ratio less than, but asymptotic to, the sharp bound for large $n$.
\end{remark}

\begin{remark}
    The proofs of Lemmas \ref{prop:8bound} and \ref{6bound} do not make use of the fact that $B$ is a Bochner tensor. In fact, the statements hold for any algebraic K\"ahler curvature tensor without modification. It is possible to use that $B$ is Bochner to obtain slightly improved bounds in some cases in the proof of Lemma \ref{6bound}, including in the case that gives the $6$, but examples show that no better dimension-independent bound is possible.
\end{remark}

\begin{corollary}
    Theorem \ref{maintheorem} is true.
\end{corollary}
\begin{proof}
    By Lemma \ref{6bound} and the observation that
    \begin{align*}
        |B(S)|^2\le |C_B|_{op}^2|S|^2 = |B_{\alpha \bar \beta \gamma \bar \delta}|^2 |S|^2
    \end{align*}
    one obtains that the constant $c$ from Lemma \ref{maximizationproblem} satisfies $c\le \frac{6+4}{4}$, where one has to keep in mind that $|B|^2=4|B_{\alpha \bar \beta \gamma \bar \delta}|^2$.
    Now Proposition \ref{maximizationproblem} gives us $\Ric_L(R, \bar R)\ge 0$, which implies $\nabla R=0$ globally via maximum principle. To check that no non-flat symmetric space except $\CP^n$ satisfies the assumptions is a quick consultation of Calabi-Vesentini's tables \cite{CV}.
\end{proof}

One may remark that the consultation of curvature tables can be avoided: After observing that strict inequality holds in Lemma \ref{6bound} for nonzero $|B|^2$, one gets strict inequality in the Bochner formula unless $B=0$. However, strict inequality in the Bochner formula is impossible by maximum principle.

\section{Examples}

Our proof of Theorem \ref{maintheorem} rests heavily on an estimate for the quantity $|SB|^2+4|B(S)|^2$ for eigenvectors $S$ of the Bochner tensor. While the estimates we give are (asymptotically) sharp for each term separately, it remains an interesting question as to whether something better holds true for the sum. In this section, we compute this quantity for some of the spaces that are natural guesses for realizing the worst case, and, based on the examples, conjecture some estimates. The examples we consider are projective spaces, their products, and complex quadrics, each with their symmetric space metric.

A table of some of the relevant quantities for each of the examples we consider is given at the end of the article as Table \ref{tab:examplevalues}.

\subsection{Complex projective space} On $\mathbb{CP}^n$ scaled to constant holomorphic curvature $\kappa$, i.e., sectional curvatures in $[\kappa/4, \kappa]$, we have $C=\kappa \id_{S^{2,0}}$, so the Bochner tensor vanishes and all related quantities are trivial.

\subsection{Products of projective spaces} The next simplest examples are products of projective spaces, $M=\mathbb{CP}^{d_1}\times \ldots \times \mathbb{CP}^{d_k}$, with their metrics scaled so that the product is Einstein. The Calabi operator splits according to $S^{2,0}M=\bigoplus S^{2,0}\mathbb{CP}^{d_i} \oplus \bigoplus_{i\ne j} T^{1,0}\mathbb{CP}^{d_i} \odot T^{1,0}\mathbb{CP}^{d_j}$ with the mixed terms constituting the kernel. 

To avoid excessive calculations, we consider the case of two factors first: 

Here we get ratios  $ \frac{|SB|^2+4|B(S)|^2}{|S|^2|B|^2}$ of $ \frac{d_j +1}{d_i (n+2)}$ and $(\frac{n+1}{d_1 d_2}-\frac{1}{n+2}\frac{(d_1+1)(d_2+1)}{d_1 d_2})$ respectively on the different eigenspaces of $B.$

In the first case, this constant is always strictly less than $1$. The constant in the second case is strictly less than $1$ for $d_1, d_2>1$. For $d_1=1$ we get $1+\frac{4}{(n-1)(n+2)}$. We will see that this ties with the other worst-case contender in the case $d_1=d_2=1$, but that in other dimensions the worst case is not attained on products of two complex projective spaces.

Since we obtain the worst constant for a product of two $\CP^1$s, another natural candidate is the product $\CP^1 \times \ldots \times \CP^1$, so we compute this as well. This also remedies the fact that actually here the two eigenvalues are equal on $S^{2,0}\CP^{d_i}$, so the eigenspaces are bigger than we computed above. Here one has the eigenspaces $\bigoplus S^{2,0}\CP^1_i$ and $\bigoplus T^{1,0}\CP^1_i\odot T^{1,0}\CP^1_j$, and the ratios are $\frac{1}{n}$ and $\frac{2}{n-1},$ respectively. 

\subsection{Complex quadrics}
On the complex quadric with its symmetric metric, the Calabi operator is given by $\frac{2}{n}\lambda \id-\lambda q\otimes q^*$, where $q$ is a symmetric $2$-tensor that diagonalizes to $q=\frac{1}{\sqrt{n}}\sum Z_\alpha\otimes Z_\alpha$.

As a straightforward consequence, $B=\tau_1 \id-\frac{n(n+1)}{2}\tau_1 q\otimes q^*$, where $\tau_1$ is determined by the scaling, and we get eigenspaces $\langle q \rangle$ and $\langle q \rangle ^\perp$.

Consulting Table \ref{tab:examplevalues}, we see that in this example $|SB|^2+4|B(S)|^2$ is maximized on the special tensor $q$ with a value of $\frac{(n+2)}{n}|B|^2|S|^2$, with the complement achieving a smaller constant. Noting that the smallest case is $n=2$, we get $|SB|^2+4|B(S)|^2\le 2|S|^2|B|^2$ on the complex quadrics.

\begin{remark}
    Computing the half-trace operator $C-\frac{\lambda}{n+1}id$, one checks that it is $k$-nonnegative exactly for $k\ge\frac{n(n+1)}{n+2}$. This is exactly the value that Lemma \ref{maximizationproblem} predicts subject to the condition $|SB|^2+4|B(S)|^2\le \frac{(n+2)}{n}|B|^2|S|^2$.
\end{remark}

\begin{conjecture}\label{optimalityconjecture}
    If $B$ is an algebraic Bochner operator, then $|SB|^2+4|B(S)|^2 \leq 2 |B|^2|S|^2$
    for any eigenvector $S$ of $B$, viewed as an operator on $S^{2,0}$.
\end{conjecture}
No better constant can asymptotically exist based just on algebraic curvature tensor computations, given the examples based on remark \ref{rem:sharpness}.

The worst case we know for actual manifolds is $\frac{n+2}{n}$. Note that the constant we obtained for products of $\CP^n$ is strictly worse with the exception of $\CP^1\times \CP^1$, where the constants are equal. If this bound could be proven, Lemma \ref{maximizationproblem} would imply

\begin{conjecture}
    Let $M$ be K\"ahler-Einstein, with Calabi operator $C$ such that $C-\frac{\lambda}{n+1}$ is $\frac{n(n+1)}{n+2}$-positive. Then $M$ is isometric to $\CP^n$, up to scaling.

    If $C-\frac{\lambda}{n+1}$ is only $\frac{n(n+1)}{n+2}$-nonnegative, then $M$ is flat, isometric to $\CP^n$, or isometric to a complex quadric.
\end{conjecture}

\begin{table}[!h]
    \centering
    \begin{tabular}{|c|c|c|c|}
        \hline
         $\CP^{d_1}\times \CP^{d_2}$ & & $\lambda=(n+1)/2$ & $|B|^2=2\frac{(n+1)(n+2)d_1 d_2}{(d_1+1)(d_2+1)}$  \\
        \hline
        \hline
        Eigenspace & Eigenvalue & $|SB|^2$ & $(|SB|^2+4|B(S)|^2) /|B|^2$ \\
        \hline
        $S^{2,0}_i$ & $\frac{n-d_i}{d_i+1}$ & $2d_j+2(\frac{n-d_i}{d_i+1})^2  (d_i-1)$ &$ \frac{d_j +1}{d_i (n+2)}$  \\
        \hline
        $T^{1,0}_1\odot T^{1,0}_2$ & $-1$ & $2(\frac{(n+1)^2(n+2)}{(d_1+1)(d_2+1)}-n-3)$ & $\frac{(n+1)(n+2)-(d_1+1)(d_2+1)}{d_1 d_2(n+2)}$  \\
        \hline
        \hline
        \hline
         $\CP^{1}\times \ldots \times \CP^{1}$ & & $\lambda=(n+1)/2$ & $|B|^2=n(n-1)(n+1)$  \\
        \hline
        \hline
        Eigenspace & Eigenvalue & $|SB|^2$ & $(|SB|^2+4|B(S)|^2) /|B|^2$  \\
        \hline
        $\bigoplus S^{2,0}_i$ & $\frac{n-1}{2}$ & $2(n-1)$ & $\frac{1}{n}$  \\
        \hline
        $\bigoplus T^{1,0}_i\odot T^{1,0}_j$ & $-1$ & $2(n-1)(n+2)$ & $\frac{2}{n-1}$  \\
        \hline\hline
        \hline
         Quadric & & $\lambda=(n+1)/2$ & $|B|^2=(n+1)(n-1)(n+2)/n$  \\
        \hline
        \hline
        Eigenspace & Eigenvalue & $|SB|^2$ & $(|SB|^2+4|B(S)|^2) /|B|^2$  \\
        \hline
        $\langle q\rangle$ & $-\frac{(n-1)(n+2)}{2n}$ & $2\frac{(n-1)(n+2)^2}{n^2}$ & $\frac{n+2}{n}$  \\
        \hline
        $\langle q \rangle^\perp$, $n\ge 2$ & $\frac{1}{n}$ & $2(n+1)-\frac{4n+8}{n^2}$ & $2\frac{n^2-2}{n(n-1)(n+2)}$  \\
        \hline
    \end{tabular}
    \caption{The Bochner tensor on examples}
    \label{tab:examplevalues}
\end{table}

\end{document}